\documentclass[12pt]{article}
\usepackage{amsmath,amsfonts,amssymb,amsthm,txfonts,hyperref,url,t1enc}
\newtheorem{theorem}{Theorem}
\newtheorem{lemma}{Lemma}
\newcommand{\N}{\mathbb N}
\newcommand{\Z}{\mathbb Z}
\begin{document}
\begin{sloppypar}
\noindent
\centerline{{\large Small systems of Diophantine equations with a prescribed}}
\vskip 0.2truecm
\noindent
\centerline{{\large number of solutions in non-negative integers}}
\vskip 0.4truecm
\noindent
\centerline{{\large Apoloniusz Tyszka}}
\vskip 0.4truecm
\noindent
{\bf Abstract.} Let \mbox{$E_n=\{x_i=1,~x_i+x_j=x_k,~x_i \cdot x_j=x_k: i,j,k \in \{1,\ldots,n\}\}$}.
If Matiyasevich's conjecture on single-fold Diophantine representations is true, then for every
computable function \mbox{$f:\N \to \N$} there is a positive integer $m(f)$ such that for
each integer \mbox{$n \geq m(f)$} there exists a system \mbox{$U \subseteq E_n$} which has
exactly $f(n)$ solutions in non-negative integers \mbox{$x_1,\ldots,x_n$}.
\vskip 0.4truecm
\noindent
{\bf 1.~~Introduction}
\vskip 0.4truecm
\par
We state the following double problem:
\vskip 0.2truecm
\noindent
$Does$ $there$ $exist$ $an$ $algorithm$ $which$ $to$ $each$ $Diophantine$ $equation$
$assigns$ $an$ $integer$ $which$ $is$ $greater$ $than$ $the$ $\stackrel{\textstyle \nearrow}{\textstyle \searrow}$
$\stackrel{\textstyle number~of~solutions}{\textstyle solutions~~~~~~~~~~~~~~~~~~~}$ $in$ \mbox{$non$-$negative$}
$integers$, $if$ $these$ $solutions$ $form$ $a$ $finite$ $set?$ $\stackrel{\textstyle ~}{\textstyle ~}$
\vskip 0.2truecm
\par
If we assume Matiyasevich's conjecture on single-fold Diophantine representations, the answer to both questions
is negative, see Theorem~\ref{main1}. If we assume the author's conjecture on integer arithmetic (\cite{Tyszka}),
the answer to both questions is positive, see Theorem~\ref{main2}.
\vskip 0.4truecm
\noindent
{\bf 2.~~Matiyasevich's conjecture vs the author's conjecture}
\vskip 0.4truecm
The Davis-Putnam-Robinson-Matiyasevich theorem states that every recursively enumerable
set \mbox{${\cal M} \subseteq {\N}^n$} has a Diophantine representation, that is
\begin{equation}
\tag*{\tt (R)}
(a_1,\ldots,a_n) \in {\cal M} \Longleftrightarrow
\exists x_1, \ldots, x_m \in \N ~~W(a_1,\ldots,a_n,x_1,\ldots,x_m)=0
\end{equation}
for some polynomial $W$ with integer coefficients, see \cite{Matiyasevich1} and \cite{Kuijer}.
The polynomial~$W$ can be computed, if we know a Turing machine~$M$
such that, for all \mbox{$(a_1,\ldots,a_n) \in {\N}^n$}, $M$ halts on \mbox{$(a_1,\ldots,a_n)$} if and only if
\mbox{$(a_1,\ldots,a_n) \in {\cal M}$}, see \cite{Matiyasevich1} and \cite{Kuijer}.
\vskip 0.2truecm
The representation {\tt (R)} is said to be \mbox{single-fold} if for any \mbox{$a_1,\ldots,a_n \in \N$} the equation
\mbox{$W(a_1,\ldots,a_n,x_1,\ldots,x_m)=0$} has at most one solution \mbox{$(x_1,\ldots,x_m) \in {\N}^m$}.
Yu.~Matiyasevich conjectures that each recursively enumerable set \mbox{${\cal M} \subseteq {\N}^n$}
has a single-fold Diophantine representation, see \mbox{\cite[pp.~341--342]{Davis}},
\mbox{\cite[p.~42]{Matiyasevich2}}, and \mbox{\cite[p.~79]{Matiyasevich3}}.
\vskip 0.2truecm
\par
Let $[\cdot]$ denote the integer part function,
\[
E_n=\{x_i=1,~x_i+x_j=x_k,~x_i \cdot x_j=x_k: i,j,k \in \{1,\ldots,n\}\}
\]
Before the main Theorem~\ref{main1}, we need an algebraic lemma together with introductory matter.
Let \mbox{$D(x_1,\ldots,x_p) \in {\Z}[x_1,\ldots,x_p] \setminus \{0\}$}.
A simple algorithm transforms the equation \mbox{$D(x_1,\ldots,x_p)=0$}
into an equivalent equation \mbox{$A(x_1,\ldots,x_p)=B(x_1,\ldots,x_p)$},
where the polynomials \mbox{$A(x_1,\ldots,x_p)$} and \mbox{$B(x_1,\ldots,x_p)$}
have non-negative integer coefficients and
\[
A(x_1,\ldots,x_p) \not\in\ \{x_1,\ldots,x_p,0\} \wedge B(x_1,\ldots,x_p) \not\in \{x_1,\ldots,x_p,0, A(x_1,\ldots,x_p)\}
\]
Let $\delta$ denote the maximum of the coefficients
of \mbox{$A(x_1,\ldots,x_p)$} and \mbox{$B(x_1,\ldots,x_p)$}, and let ${\cal T}$ denote the family of all polynomials
$W(x_1,\ldots,x_p) \in {\Z}[x_1,\ldots,x_p]$ whose coefficients belong to the interval $[0,~\delta]$ and
\[
{\rm deg}(W,x_i) \leq {\rm max}\Bigl({\rm deg}(A,x_i),~{\rm deg}(B,x_i)\Bigr)
\]
for each \mbox{ $i \in \{1,\ldots,p\}$}. Here we consider the degrees with respect to the variable~$x_i$.
Let $n$ denote the cardinality of~${\cal T}$. We choose any bijection
\[
\tau: \{p+1,\ldots,n\} \longrightarrow {\cal T} \setminus \{x_1,\ldots,x_p\}
\]
such that $\tau(p+1)=0$, $\tau(p+2)=A(x_1,\ldots,x_p)$, and $\tau(p+3)=B(x_1,\ldots,x_p)$.
Let ${\cal H}$ denote the family of all equations of the form
\[
x_i=1,~~x_i+x_j=x_k,~~x_i \cdot x_j=x_k~~(i,j,k \in \{1,\ldots,n\})
\]
which are polynomial identities in \mbox{${\Z}[x_1,\ldots,x_p]$} if
\[
\forall s \in \{p+1,\ldots,n\} ~~x_s=\tau(s)
\]
Since $\tau(p+1)=0$, the equation $x_{p+1}+x_{p+1}=x_{p+1}$ belongs to ${\cal H}$.
Let
\[
S={\cal H} \cup \{x_{p+1}+x_{p+2}=x_{p+3}\}
\]
\begin{lemma}\label{lem1}
The system~$S$ can be computed, \mbox{$S \subseteq E_n$}, and
\[
\forall x_1,\ldots,x_p \in \N ~\Bigl(D(x_1,\ldots,x_p)=0 \Longleftrightarrow
\]
\[
\exists x_{p+1},\ldots,x_n \in \N ~(x_1,\ldots,x_p,x_{p+1},\ldots,x_n) {\rm ~solves~} S\Bigr)
\]
For each \mbox{$x_1,\ldots,x_p \in \N$} with
\mbox{$D(x_1,\ldots,x_p)=0$} there exists a unique tuple \mbox{$(x_{p+1},\ldots,x_n) \in {\N}^{n-p}$}
such that the tuple \mbox{$(x_1,\ldots,x_p,x_{p+1},\ldots,x_n)$} solves $S$.
Hence, the equation \mbox{$D(x_1,\ldots,x_p)=0$} has the same number of non-negative integer
solutions as $S$.
\end{lemma}
\begin{theorem}\label{main1}
If Matiyasevich's conjecture is true, then for every computable function
\mbox{$f:\N \to \N$} there is a positive integer $m(f)$ such that for each integer
\mbox{$n \geq m(f)$} there exists a system \mbox{$U \subseteq E_n$} which has exactly $f(n)$
solutions in non-negative integers \mbox{$x_1,\ldots,x_n$}.
\end{theorem}
\begin{proof}
By Matiyasevich's conjecture, there is a non-zero polynomial \mbox{$W(x_1,x_2,x_3,\ldots,x_r)$}
with integer coefficients such that for each non-negative integers $x_1$, $x_2$,
\[
x_1=f(x_2) \Longleftrightarrow \exists x_3, \ldots, x_r \in \N ~~W(x_1,x_2,x_3,\ldots,x_r)=0
\]
and at most one tuple \mbox{$(x_3,\ldots,x_r) \in {\N}^{r-2}$} satisfies \mbox{$W(x_1,x_2,x_3,\ldots,x_r)=0$}.
By Lemma~\ref{lem1}, there is an integer \mbox{$s \geq 3$} such that for each non-negative integers $x_1$, $x_2$,
\begin{equation}
\tag*{${\tt (E)}$}
x_1=f(x_2) \Longleftrightarrow \exists x_3,\ldots,x_s \in \N ~~\Psi(x_1,x_2,x_3,\ldots,x_s)
\end{equation}
where the formula \mbox{$\Psi(x_1,x_2,x_3,\ldots,x_s)$} is algorithmically determined as a conjunction of formulae of
the form \mbox{$x_i=1$}, \mbox{$x_i+x_j=x_k$}, \mbox{$x_i \cdot x_j=x_k$} \mbox{($i,j,k \in \{1,\ldots,s\}$)} and
\par
\begin{description}
\item{${\tt (SF)}$}
for each non-negative integers $x_1$, $x_2$, at most one tuple
\mbox{$(x_3,\ldots,x_s) \in {\N}^{s-2}$} satisfies \mbox{$\Psi(x_1,x_2,x_3,\ldots,x_s)$}.
\end{description}
Let \mbox{$m(f)=12+2s$}. If \mbox{$n \geq m(f)$} and \mbox{$f(n)=0$}, then we put \mbox{$U=E_n$}.
Assume that \mbox{$n \geq m(f)$} and \mbox{$f(n) \geq 1$}. For each integer \mbox{$n \geq m(f)$},
\[
n-\left[\frac{n}{2}\right]-6-s \geq m(f)-\left[\frac{m(f)}{2}\right]-6-s \geq m(f)-\frac{m(f)}{2}-6-s=0
\]
Let $U$ denote the following system
\[\left\{
\begin{array}{rcl}
{\rm all~equations~occurring~in~}\Psi(x_1,x_2,x_3,\ldots,x_s) \\
n-\left[\frac{n}{2}\right]-6-s {\rm ~equations~of~the~form~} z_i=1 \\
t_1 &=& 1 \\
t_1+t_1 &=& t_2 \\
t_2+t_1 &=& t_3 \\
&\ldots& \\
t_{\left[\frac{n}{2}\right]-1}+t_1 &=& t_{\left[\frac{n}{2}\right]} \\
t_{\left[\frac{n}{2}\right]}+t_{\left[\frac{n}{2}\right]} &=& w \\
w+y &=& x_2 \\
y+y &=& y {\rm ~(if~}n{\rm ~is~even)} \\
y &=& 1 {\rm ~(if~}n{\rm ~is~odd)} \\
t &=& 1 \\
z+t &=& x_1 \\
u+v &=& z \\
\end{array}
\right.\]
with $n$ variables. By the equivalence ${\tt (E)}$, the system $U$ is consistent over $\N$.
If a $n$-tuple \mbox{$(x_1,x_2,x_3,\ldots,x_s,\ldots,w,y,t,z,u,v)$} consists of non-negative
integers and solves $U$, then by the equivalence ${\tt (E)}$,
\[
x_1=f(x_2)=f(w+y)=f\left(2 \cdot \left[\frac{n}{2}\right]+y\right)=f(n)
\]
Hence, the last three equations in $U$, together with statements \mbox{${\tt (E)}$} and \mbox{${\tt (SF)}$},
guarantee us that the system $U$ has exactly $f(n)$ solutions in non-negative integers.
\end{proof}
\vskip 0.2truecm
\par
The following Conjecture contradicts to Matiyasevich's conjecture, see \cite{Tyszka}.
\vskip 0.2truecm
\noindent
{\bf Conjecture}~(\cite{Tyszka},~\cite{Cipu}). {\em If a system \mbox{$S \subseteq E_n$} has only finitely
many solutions in integers \mbox{$x_1,\ldots,x_n$}, then each such solution \mbox{$(x_1,\ldots,x_n)$}
satisfies \mbox{$|x_1|,\ldots,|x_n| \leq 2^{\textstyle 2^{n-1}}$}.}
\newpage
\noindent
{\bf Observation.} {\em For $n \geq 2$, the bound $2^{\textstyle 2^{n-1}}$ cannot be decreased because the system
\begin{displaymath}
\left\{
\begin{array}{rcl}
x_1+x_1 &=& x_2 \\
x_1 \cdot x_1 &=& x_2 \\
x_2 \cdot x_2 &=& x_3 \\
x_3 \cdot x_3 &=& x_4 \\
&\ldots& \\
x_{n-1} \cdot x_{n-1} &=&x_n
\end{array}
\right.
\end{displaymath}
\noindent
has exactly two integer solutions, namely $\left(0,\ldots,0\right)$ and
$\left(2,4,16,256,\ldots,2^{\textstyle 2^{n-2}},2^{\textstyle 2^{n-1}}\right)$.}
\vskip 0.2truecm
\par
Every Diophantine equation of degree at most $n$ has the form
\begin{equation}\label{equ1}
\sum_{\stackrel{\textstyle i_1,\ldots,i_k \in \{0,\ldots,n\}}{\textstyle i_1+\ldots+i_k \leq n}}
a(i_1,\ldots,i_k) \cdot x_1^{\textstyle i_1} \cdot \ldots \cdot x_k^{\textstyle i_k}=0
\end{equation}
where $a(i_1,\ldots,i_k)$ denote integers.
\begin{theorem}\label{main2} (\cite{Tyszka})
The Conjecture implies that if a Diophantine equation~(\ref{equ1}) has only finitely many solutions
in integers (non-negative integers, rationals), then their heights are bounded from above by
a computable function of
\[
{\rm max}\Bigl(\{k,n\} \cup \Bigl\{|a(i_1,\ldots,i_k)|:~\bigl(i_1,\ldots,i_k \in \{0,\ldots,n\}\bigr) \wedge \bigl(i_1+\ldots+i_k \leq n\bigr)\Bigr\}\Bigl)
\]
\end{theorem}
\vskip 0.2truecm
\noindent
{\bf Corollary.} {\em The conclusion of Theorem~\ref{main1} and the Conjecture are jointly inconsistent.}
\vskip 0.4truecm
\noindent
{\bf 3.~~Without unproven assumptions}
\vskip 0.4truecm
\begin{theorem}\label{the2}
For each integer \mbox{$n \geq 2$} and each integer \mbox{$m \geq 3+2 \cdot \left[\log_2\left(n-1\right)\right]$}
there exists a system \mbox{$U \subseteq \{x_i=1,~x_i+x_j=x_k: i,j,k \in \{1,\ldots,m\}\}$}
which has exactly $n$ solutions in non-negative integers \mbox{$x_1,\ldots,x_m$}.
\end{theorem}
\begin{proof}
The equation \mbox{$x+y=n-1$} has exactly $n$ solutions in non-negative integers \mbox{$x,y$}.
In order to write the equation \mbox{$x+y=n-1$} as an equivalent system of
equations of the form \mbox{$x_i=1$}, \mbox{$x_i+x_j=x_k$}, we assign new variables to
$1$, $x$, \mbox{and $y$}. Starting from $1$, we can compute $n-1$ by performing at most \mbox{$2 \cdot \left[\log_2\left(n-1\right)\right]$}
additions. Applying this observation, we find an equivalent system which contains at most
\mbox{$3+2\cdot\left[\log_2\left(n-1\right)\right]$} variables. If the found system contains
less than $m$ variables, then we declare that the missing variables are equal to $1$.
\end{proof}
\begin{theorem}\label{the3}
For each positive integer $n$ and each integer \mbox{$m \geq 11+2 \cdot \left[\log_2\left(2n-1\right)\right]$}
there exists a system \mbox{$U \subseteq E_m$} which has exactly $n$ solutions in non-negative
integers \mbox{$x_1,\ldots,x_m$} and at most finitely many solutions in integers \mbox{$x_1,\ldots,x_m$}.
\end{theorem}
\begin{proof}
For each non-negative integer $n$, the equation \mbox{$(2x+1)^2+(2y)^2=5^{\textstyle 2n-1}$} has exactly $n$
solutions in non-negative integers, see \cite{Schinzel}. In order to write the
equation \mbox{$(2x+1)^2+(2y)^2=5^{\textstyle 2n-1}$} as an equivalent system of equations of the form
\mbox{$x_i=1$}, \mbox{$x_i+x_j=x_k$}, \mbox{$x_i \cdot x_j=x_k$}, we assign new variables to
\[
1,~~2,~~3,~~5,~~x,~~x+1,~~2x+1,~~(2x+1)^2,~~y,~~2y,~~(2y)^2
\]
Starting from $5$, we can compute \mbox{$5^{\textstyle 2n-1}$} by performing at most
\mbox{$2 \cdot \left[\log_2\left(2n-1\right)\right]$} multiplications.
Applying this observation, we find an equivalent system which contains at
most \mbox{$11+2 \cdot \left[\log_2\left(2n-1\right)\right]$} variables.
If the found system contains less than $m$ variables, then we declare that the missing
variables are equal to $1$.
\end{proof}
\begin{theorem}\label{the4}
For each integer \mbox{$n \geq 4$} and each integer \mbox{$m \geq 8+2\cdot\left[\log_2\left(n-3\right)\right]$}
there exists a system \mbox{$U \subseteq E_m$} which has exactly $n$ solutions in integers \mbox{$x_1,\ldots,x_m$}.
\end{theorem}
\begin{proof}
Let
\begin{displaymath}
D(t,x,y)=\left\{
\begin{array}{cl}
x \cdot y - 2^{\textstyle \frac{t-2}{2}} & {\rm ~~if~~}t{\rm ~~is~~even} \\
\left(x \cdot y - 2^{\textstyle \frac{t-3}{2}}\right) \cdot \left(x^2+y^2\right) & {\rm ~~if~~}t{\rm ~~is~~odd}
\end{array}
\right.
\end{displaymath}
For each non-negative integer $n$, the equation \mbox{$D(n,x,y)=0$} has exactly $n$ integer solutions.
\vskip 0.2truecm
\noindent
{\em Case~1:}~$n$~$\in$~$\{4,6,8,\ldots\}$. In order to write the equation \mbox{$x \cdot y=2^{\textstyle \frac{n-2}{2}}$}
as an equivalent system of equations of the form $x_i=1$, \mbox{$x_i+x_j=x_k$}, \mbox{$x_i \cdot x_j=x_k$},
we assign new variables to $1$, $2$, $x$, $y$. Starting from $2$, we can compute \mbox{$2^{\textstyle \frac{n-2}{2}}$}
by performing at most \mbox{$2 \cdot \left[\log_2\left(\frac{n-2}{2}\right)\right]$} multiplications.
Applying this observation, we find an equivalent system which contains at
most \mbox{$4+2 \cdot \left[\log_2\left(\frac{n-2}{2}\right)\right]$} variables. Since
\[
4+2 \cdot \left[\log_2\left(\frac{n-2}{2}\right)\right]=
2+2 \cdot \left[\log_2\left(n-2\right)\right]<8+2 \cdot \left[\log_2\left(n-3\right)\right] \leq m
\]
the found system contains less than $m$ variables. We declare that the missing variables are equal to $1$.
\vskip 0.2truecm
\noindent
{\em Case~2:}~$n$~$\in$~$\{5,7,9,\ldots\}$. In order to write the equation
\mbox{$\left(x \cdot y - 2^{\textstyle \frac{n-3}{2}}\right) \cdot \left(x^2+y^2\right)=0$}
as an equivalent system of equations of the form \mbox{$x_i=1$},
\mbox{$x_i+x_j=x_k$}, \mbox{$x_i \cdot x_j=x_k$}, we assign new variables to
\[
1,~~2,~~x,~~y,~~x \cdot y,~~x \cdot y-2^{\textstyle \frac{n-3}{2}},~~x^2,~~y^2,~~x^2+y^2,
~~\left(x \cdot y - 2^{\textstyle \frac{n-3}{2}}\right) \cdot \left(x^2+y^2\right)
\]
Starting from $2$, we can compute \mbox{$2^{\textstyle \frac{n-3}{2}}$} by performing
at most \mbox{$2 \cdot \left[\log_2\left(\frac{n-3}{2}\right)\right]$} multiplications.
Applying this observation, we find an equivalent system which contains at
most \mbox{$10+2 \cdot \left[\log_2\left(\frac{n-3}{2}\right)\right]=8+2\cdot \left[\log_2\left({n-3}\right)\right]$}
variables. If the found system contains less than $m$ variables, then we declare that the missing
variables are equal to $1$.
\end{proof}
\vskip 0.2truecm
Let
\[
D(x,u,v,s,t)=\left(u+v-x+1\right)^2+\left(2^u-s\right)^2+\left(2^v-t\right)^2
\]
For each non-positive integer $k$, the equation \mbox{$D(k,u,v,s,t)=0$} has no integer solutions.
For each positive integer $k$, the equation \mbox{$D(k,u,v,s,t)=0$} has exactly $k$ integer solutions.
\vskip 0.2truecm
\par
Let
\[
D(x,u,v,s,t)=8(u^2+v^2+s^2+t^2+1)-x
\]
For each non-positive integer $k$, the equation \mbox{$D(k,u,v,s,t)=0$} has no integer
solutions. Jacobi's four-square theorem says that for each positive integer $k$ the number
of representations of $k$ as a sum of four squares of integers equals $8s(k)$, where $s(k)$
is the sum of positive divisors of $k$ which are not divisible by $4$, \mbox{see \cite{Hirschhorn}}.
By Jacobi's theorem, for each prime $p$ the equation \mbox{$D(8(p+1),u,v,s,t)=0$} has exactly
\mbox{$8(p+1)$} integer solutions.
\vskip 0.2truecm
\noindent
{\bf Open Problem.} {\em Does there exist a polynomial \mbox{$D(x,x_1,\ldots,x_n)$}
with integer coefficients such that for each non-positive integer $k$ the equation
\mbox{$D(k,x_1,\ldots,x_n)=0$} has no integer solutions and for each positive integer
$k$ the equation \mbox{$D(k,x_1,\ldots,x_n)=0$} has exactly $k$ integer solutions?}
\vskip 0.2truecm
\par
Let the polynomials \mbox{$P_k(x)$} are defined by the recurrence \mbox{$P_{k+1}(x)=4P_k(x)(1-P_k(x))$}
with \mbox{$P_0(x)=x$}.
\begin{lemma}\label{lem2}
For each non-negative integer $k$, the polynomial \mbox{$1-2P_k(x)$} has exactly $2^k$ distinct real roots.
\end{lemma}
\begin{proof}
A brief calculation leads to the equality
\begin{equation}\label{equ2}
1-2P_k(x)={\rm cos}(2^k \cdot {\rm arccos}(1-2x))
\end{equation}
where \mbox{$x \in [0,1]$} and $k$ ranges over non-negative integers, see \cite[p.~60]{Schuster}.
For \mbox{$i \in \{0,\ldots,2^k-1\}$}, the numbers
\[
\frac{1-{\rm cos}\left(\frac{\textstyle 4i+1}{\textstyle 2^{k+1}} \cdot \pi \right)}{2}
\]
are pairwise different and belong to $(0,1)$. Equation~(\ref{equ2}) implies that these numbers solve the
equation \mbox{$1-2P_k(x)=0$}. Since ${\rm deg}(1-2P_k(x))=2^k$, no other solutions exist.
\end{proof}
\begin{theorem}\label{the5}
For each positive integer $n$ there exists a positive integer $\tau(n)$ and a~system
\mbox{$U \subseteq E_{\tau(n)}$} such that $U$ has exactly $n$ solutions in reals
\mbox{$x_1,\ldots,x_{\tau(n)}$} and $\tau(n)$ grows linearly with \mbox{$\left[\log_2\left(n\right)\right]$}.
\end{theorem}
\begin{proof}
Let
\[
n=\sum_{\textstyle k=0}^{\textstyle \left[\log_2\left(n\right)\right]} a_k \cdot 2^k
\]
where all \mbox{$a_k \in \{0,1\}$}, and let
\[
W_n(x,y)=
\prod_{\stackrel{\textstyle k \in \left\{0,\ldots,\left[\log_2\left(n\right)\right]\right\}}{\textstyle a_k=1}}
\left(1-2P_k(x)\right)^2+\left(y-k\right)^2
\]
By Lemma~\ref{lem2}, for each \mbox{$k \in \left\{0,\ldots,\left[\log_2\left(n\right)\right]\right\}$} the equation
\mbox{$\left(1-2P_k(x)\right)^2+\left(y-k\right)^2=0$} has exactly $2^k$ real solutions \mbox{$(x,y)$}
and each of them satisfies \mbox{$y=k$}. Therefore, the equation \mbox{$W_n(x,y)=0$}
has exactly
\[
\sum_{\stackrel{\textstyle k \in \left\{0,\ldots,\left[\log_2\left(n\right)\right]\right\}}{\textstyle a_k=1}} 2^k=n
\]
real solutions. From the definition of polynomials \mbox{$P_k(x)$}, it follows that the equation \mbox{$W_n(x,y)=0$}
can be equivalently written as a system of equations of the form \mbox{$x_i=1$}, \mbox{$x_i+x_j=x_k$}, \mbox{$x_i \cdot x_j=x_k$}
in such a way that the number of variables grows linearly with \mbox{$\left[\log_2\left(n\right)\right]$}.
\end{proof}

\noindent
Apoloniusz Tyszka\\
Technical Faculty\\
Hugo Ko\l{}\l{}\k{a}taj University\\
Balicka 116B, 30-149 Krak\'ow, Poland\\
E-mail address: \url{rttyszka@cyf-kr.edu.pl}
\end{sloppypar}
\end{document}